\documentclass[11pt,oneside,english]{amsart}
\usepackage{color}
\usepackage{textcomp}
\usepackage{amstext}
\usepackage{amsthm}
\usepackage{amssymb}
\usepackage{geometry}
\makeatletter
\numberwithin{equation}{section}
\numberwithin{figure}{section}
\theoremstyle{plain}
\newtheorem{thm}{\protect\theoremname}
\theoremstyle{plain}
\newtheorem{lem}[thm]{\protect\lemmaname}
\theoremstyle{plain}
\newtheorem{cor}[thm]{\protect\corollaryname}

\makeatother

\usepackage{babel}
\providecommand{\corollaryname}{Corollary}
\providecommand{\lemmaname}{Lemma}
\providecommand{\theoremname}{Theorem}

\begin{document}
\title[Compactness of complete $k$-curavture metrics]{Compactness theorem of complete $k$-curvature manifolds with isolated
singularities}
\author{Wei Wei}
\address{Shanghai Center for Mathematical Science, Fudan University, Shanghai,
China}
\email{weisx001@mail.ustc.edu.cn}
\begin{abstract}
In this paper we prove that the set of metrics conformal to the standard
metric on $\mathbb{S}^{n}\backslash\{p_{1},\cdots,p_{l}\}$ is locally compact
in $C^{m,\alpha}$ topology for any $m>0$, whenever the metrics have
constant $\sigma_{k}$ curvature and the $k$-Dilational Pohozaev
invariants have positive lower bound for $k<n/2$. Here the $k$-Dilational
Pohozaev invariants come from the Kazdan-Warner type identity
for the $\sigma_{k}$ curvature, which is derived by Viaclovsky \cite{Viac2000}
and Han \cite{H1}. When $k=1$, Pollack \cite{Pollack} proved the
compactness results for the complete metrics of constant positive
scalar curvature on $\mathbb{S}^{n}\backslash\{p_{1},\cdots,p_{l}\}$.
\end{abstract}

\subjclass[2000]{53A30·53C20}

\maketitle

\section{introduction}

Let $(M^{n},g)$ be a Riemannian manifold with Ricci curvature $Ric$.
Let $R$ and $A$ be the corresponding Scalar curvature and Schouten
tensor defined by 
\begin{align*}
A & =\frac{1}{n-2}(\text{{\rm Ric}}-\frac{R}{2(n-1)}g),\\
R & =g^{ij}{\rm Ric}_{ij},
\end{align*}
respectively. Let $\{\lambda(A)\}_{i=1}^{n}$ be the set of eigenvalues
of $A$ with respect to $g$. Define

\[
\sigma_{k}(g^{-1}A_{g}):=\sum_{1\le i_{1}\cdots<i_{k}\le n}\lambda_{i_{1}}\lambda_{i_{2}}\cdots\lambda_{i_{k}}.
\]
Here $\sigma_{1}(\lambda(A))={\rm Tr}A=\frac{1}{2(n-1)}R$. 

We sometimes call $(M^n,g)$ a $k$-curvature manifold if the following
equation holds
\begin{equation}
\sigma_{k}(g^{-1}A_{g})=c.\label{eq:equation}
\end{equation}

Like the classical Yamabe problem, Viaclovsky proposed the $\sigma_{k}$
Yamabe problem as follows. In the conformal class $[g]=\{e^{2u}g|u\in C^{\infty}(M)\}$,
can we find $g_{1}\in[g]$ such that $\sigma_{k}(g_{1}^{-1}A_{g_{1}})=c$
for some positive constant $c$?

Define a positive cone condition  
\[
\Gamma_{k}^{+}:=\{\lambda=(\lambda_{1},\cdots,\lambda_{n}),\ s.t.\ \sigma_{1}(\lambda)>0,\cdots,\,\sigma_{k}(\lambda)>0\}.
\]
Note that when $g^{-1}A_{g}\in\Gamma_{k}^{+}$, equation (\ref{eq:equation})
is a fully nonlinear elliptic partial differential equation for $k>1.$
In the past decades, the $k$-Yamabe problem for smooth manifolds
has been widely studied  and mathematicians have
established many parallel results to the classical Yamabe problem.
For more information, we refer to \cite{CGY1,CGY2,GV0,GV4,GW1,L1,LL1,LL2,LN,STW,TW,W1}
and references therein. 

The singular sets of locally conformally flat metrics with positive
$\sigma_{k}$ curvatures are widely studied. For $k=1$, Schoen and
Yau have proved that $\dim(\partial\Omega)<(n-2)/2$ if $\Omega\subset \mathbb{S}^{n}(n\ge5)$
admits a complete conformal metric $g$ with constant curvature. In
\cite{CHY2} Chang-Han-Yang have proved that if $\Omega\subset \mathbb{S}^{n}(n\ge5)$
admits a complete, conformal metric $g$ such that $\sigma_{1}(A_{g})\ge c>0,\sigma_{2}(A_{g})\ge0$
and $|R_{g}|+|\nabla_{g}R|_{g}\le c_{0}$, then $\dim(\partial\Omega)<(n-4)/2.$
For $1\le k<n/2$, Gonz\'{a}lez \cite{G1} and Guan-Lin-Wang
\cite{GLW} have proved that under some natural assumptions for $\sigma_{k},$
$\dim(\partial\Omega)<(n\text{\textminus}2k)/2.$ When $k>n/2$, Gonz\'{a}lez
\cite{G1} has proved that there exists no complete manifold with
$\sigma_{1}(g^{-1}A_{g}),\cdots\sigma_{k}(g^{-1}A_{g})>C_{0}$ on
subdomain of $\mathbb{S}^{n}$. 

For the singular set $P=\{p_{1},\cdots,p_{q}\}$, a classical problem
is the existence of complete metrics on $M$ which is conformal to $g$ on
$M\backslash P$ with constant scalar curvature.
This problem has been widely studied and we address the readers to the classical
paper of Schoen \cite{Sc2} and its references. The existence of complete
manifolds with constant $\sigma_{k}$ curvature attracted many mathematicians.
Li \cite{L} has proved that on $\mathbb{R}^{n}\backslash\{0\}$,
the conformal factor is radial for $1\le k\le n$. Chang-Han-Yang \cite{CHY} have proved that 
there exists a complete metric with constant positive $k$-curvature for $k<n/2$ on $\mathbb{R}^{n}\backslash\{0\}$.

Catino-Mazzieri \cite{MN} constructed complete metrics with constant
positive $k$-curvature on connected sum $M_{1}\sharp M_{2}$ for
$k<\frac{n}{2}$ and Mazzieri-Segatti \cite{LSe} constructed
a complete locally conformally flat metric with constant positive
$k$-curvature for $4\le2k<n$. In \cite{SS} Santos has proved that
when $k=2$, $n\ge5$, there exist complete metrics on $M\backslash P$
with constant positive $k$-curvature, where $(M,g_{0})$ has constant
$k$-curvatures and $\nabla_{g_{0}}^{j}W_{g_{0}}(P)=0$ for $j=0,1,\cdots,[\frac{n-4}{2}]$.
It is natural to study the compactness
of the complete metrics with positive constant $k$-curvature on
$\mathbb{S}^{n}\backslash P:=\Omega.$ When $k=n/2$, the related results are
not so abundant. For $k=2$ and $n=4$, the authors \cite{FW} have given
a necessary condition for the existence of a conic 4-sphere with positive
constant 2-curvature. For $k>\frac{n}{2}$, singularities have been
studied in \cite{L,GV3,TW09}. For non-isolated singularities, we
refer to \cite{GMM} for recent developments. 

In this paper, we want to prove that the compactness
of the complete metrics on $\mathbb{S}^{n}\backslash P$ holds under some condition.We state our
main theorem as below:
\begin{thm}
\label{thm:Main thm}Let $P=\{p_{1},\cdots,p_{q}\}\subset \mathbb{S}^{n}$
be a set of  distinct points and we assume that $k<n/2$. Suppose that
there exists a sequence $g_{i}=u_{i}^{\frac{4}{n-2}}g_{0}$ on $\mathbb{S}^{n}\backslash P$
with positive constant $\sigma_{k}$ curvature. If there exists a
point $p\in P$ such that the k-Dilational Pohozaev invariants $D_{k}(g_{i},p)$
are uniformly bounded away from $0$, then there exists a subsequence
of $\{u_{i}\}$ converging to a positive solution in $C^{\infty}$ on any compact subsets
of $\mathbb{S}^{n}\backslash P$. Moreover, the corresponding
metric is complete on $\mathbb{S}^{n}\backslash P$. Here $g_{0}$ is the standard
sphere metric.
\end{thm}

The nonvanishing $k$-Dilational Pohozaev invariants are used to prove
that the singularities are non-removable. One nonvanishing $k$-Dilational
Pohozaev invariant is enough to imply the completeness of the limiting
metric. The definition of $k$-Dilational Pohozaev invariants is the
natural extension of the Pohozaev invariants in the classical Yamabe
problem, which will be introduced in section 2. By Han-Li-Teixeira's
theorem \cite{HLT-1} and Chang-Han-Yang's radial classification \cite{CHY},
the metric on $\mathbb{S}^{n}\backslash P$ with positive constant $\sigma_{k}$
curvature is complete for $k<\frac{n}{2}$. But for $k=\frac{n}{2},$
the metric is conic. 

This paper is organized as follows. In section 2, we give some basic
notations and introduce the definition of the $k$-Dilational Pohozaev
invariants. In section 3, we give the growth estimate of the solution
near the singularity by the classical blow up analysis and the moving
plane method. In section 4, using the $k$-Dilational Pohozaev invariants,
we prove Theorem \ref{thm:Main thm} by the compactness argument.
In appendix, we will compute the value of the $k$-Dilational Pohozaev
invariants. In this paper, the constant $C,c$ may differ from line
to line without confusion. We always refer $g_{E}$ to the standard
Euclidean metric, $g_{0}$ to the standard sphere metric and $g_{u}=u^{\frac{4}{n-2}}g_{0}$.

\textit{Acknowledge:} The author would like to thank Professor Hao
Fang for helpful discussions and  Professor
Xinan Ma for constant support. The author also would like to thank
the referee for the helpful comments. 

\section{Basic Notations}

In this section, following the work of Pollack \cite{Pollack}, we
can define the $k$-Dilational Pohozaev invariants as the classical
Dilational Pohozaev invariants, which are derived from the Kazdan-Warner
type identity.

The following Kazdan-Warner type identity is proved by Viaclovsky
\cite{Viac2000} and Han\cite{H1}, which is significant in this paper.

\begin{thm}
\label{thm:k-z}\cite{H1,Viac2000}Let $(N^{n},g)$ be a locally conformally
flat $n$-dimensional manifold with boundary $\partial N$. For any
conformal Killing field $X$ on $N^{n},$ we have 
\[
\frac{n-k}{n}\int_{N}<X,\nabla\sigma_{k}(g^{-1}A)>dv_{g}=\int_{\partial N}\mathring{H}_{a}^{b}\nu_{b}X^{a}d\sigma_{g},
\]
where 
\[
\mathring{H}_{a}^{b}=H_{a}^{b}-\frac{H_{c}^{c}}{n}\delta_{a}^{b},
\]
 
\[
H_{a}^{b}=T_{k-1,c}^{b}A_{a}^{c},
\]
 
\[
T_{k-1,c}^{b}=\frac{1}{(k-1)!}\sum_{i_{1},\cdots i_{k-1},j_{1},\cdots j_{k-1}=1}^{n}\delta\left(\begin{array}{ccc}
i_{1} & \cdots i_{k-1} & b\\
j_{1} & \cdots j_{k-1} & c
\end{array}\right)A_{i_{1}}^{j_{1}}\cdots A_{i_{k-1}}^{j_{k-1}},
\]
 $\nu$ is the unit outward normal vector to $\partial N$, $dv_{g}$,
$d\sigma_{g}$ are the volume and surface measure respectively.
\end{thm}

The Kazdan-Warner type identity is derived from the divergence structure
as below
\[
\frac{n-k}{n}<X,\nabla\sigma_{k}(g^{-1}A)>=\nabla_{b}(X^{a}\mathring{H}_{a}^{b}).
\]
 For $k=1$, it has been proved by Schoen \cite{Sc2} and applied
to prove the existence of the singular Yamabe problem. The  Kazdan-Warner identity was also
used to the Yamabe problem on star-shaped domain by Pohozaev, which
is  called Pohozaev identity. For $k=n$, we refer to \cite{H1} for
the Kazdan-Warner type identity, which plays a significant role in
the Yamabe problem. 

If $g=g_{u}=u^{\frac{4}{n-2}}g_{b}$ holds for a positive function
$u$ and a smooth metric $g_{b}$, we have

\[
A=A_{g_{b}}-\frac{2}{n-2}u^{-1}\nabla_{g_{b}}^{2}u+\frac{2n}{(n-2)^{2}}u^{-2}du\otimes du-\frac{2}{(n-2)^{2}}u^{-2}|du|_{g_{b}}^{2}g_{b}.
\]

The corresponding equation of (\ref{eq:equation}) is 
\begin{equation}
\sigma_{k}\bigg(g_{b}^{-1}(A_{g_{b}}-\frac{2}{n-2}u^{-1}\nabla_{g_{b}}^{2}u+\frac{2n}{(n-2)^{2}}u^{-2}du\otimes du-\frac{2}{(n-2)^{2}}u^{-2}|du|_{g_{b}}^{2}g_{b})\bigg)=\binom{n}{k}(\frac{1}{2})^{k}u^{\frac{4k}{n-2}}.\label{eq:equation2}
\end{equation}
Here we always normalize the constant $c$ to be $\binom{n}{k}(\frac{1}{2})^{k}.$ 

When $g_{b}=g_{E}$, (\ref{eq:equation2}) becomes 

\begin{equation}
\sigma_{k}(-\frac{2}{n-2}u^{-1}\nabla_{E}^{2}u+\frac{2n}{(n-2)^{2}}u^{-2}du\otimes du-\frac{2}{(n-2)^{2}}u^{-2}|du|_{g_{E}}^{2}g_{E}))=\binom{n}{k}(\frac{1}{2})^{k}u^{\frac{4k}{n-2}}.\label{eq:euclidean equation}
\end{equation}

Consider $g=\overline{u}(t)^{\frac{4}{n-2}}(dt^{2}+d\theta^{2})$, where
$\sigma_{k}(g^{-1}A_{g})=\binom{n}{k}(\frac{1}{2})^{k}$,
and $d\theta^{2}$ is the standard $(n-1)$-sphere metric. As indicated
in Theorem 1 in \cite{CHY}, $\bar{u}$ satisfies 

\begin{equation}
\frac{1}{2}=\left(1-(\frac{2}{n-2})^{2}(\frac{u_{t}}{u})^{2}\right)^{k-1}\left(\frac{k}{n}\frac{2}{n-2}(-\frac{u_{tt}}{u}+\frac{u_{t}^{2}}{u^{2}})+(\frac{1}{2}-\frac{k}{n})\bigg(1-(\frac{2}{n-2})^{2}\frac{u_{t}^{2}}{u^{2}}\bigg)\right)u^{-2k\frac{2}{n-2}},\label{eq:ode in cylinder}
\end{equation}
and 
\[
\left[1-\overline{u}^{-\frac{4k}{n-2}}\left(1-(\frac{2}{n-2})^{2}(\frac{\overline{u}_{t}}{\overline{u}})^{2}\right)^{k}\right]\overline{u}^{\frac{2n}{n-2}}
\]
 is a nonpositive constant function.

We denote $\bar{u}_{h}$ as the radial solution $\bar{u}$ satisfying
\begin{equation}
h=\left[1-\overline{u}^{-\frac{4k}{n-2}}\left(1-(\frac{2}{n-2})^{2}(\frac{\overline{u}_{t}}{\overline{u}})^{2}\right)^{k}\right]\overline{u}^{\frac{2n}{n-2}}.\label{eq:u_h definition}
\end{equation}
The metric $g=\overline{u}(t)^{\frac{4}{n-2}}(dt^{2}+d\theta^{2})$
is the standard sphere metric when $h=0.$

To study the local behavior of the solution near singularities, the
following H\"{o}lder regularity is important.
\begin{thm}
\label{Holder regularity} \cite{HLT-1} Suppose $g=u^{\frac{4}{n-2}}(dt^{2}+d\theta^{2})$
and \textup{$\sigma_{k}(g^{-1}A_{g})=c$} on $\{t\ge t_0\}\times \mathbb{S}^{n-1}$
with $\lambda(A_{g})\in\Gamma_{k}^{+}$, where $c$ is a positive
constant and $2\le k\le n$. Then there exist positive constants $\alpha$
and $C$ such that for $t>t_0+1$, we have
\[
|u(t,\theta)-\bar{u}(t)|\le Ce^{-\alpha t}\bar{u}(t),
\]
where $\bar{u}$ is a radial smooth solution to (\ref{eq:ode in cylinder})
on $\mathbb{R}\times \mathbb{S}^{n-1}$ in the $\Gamma_{k}^{+}$ class.
\end{thm}

Chang-Han-Yang have classified the radial solution
$\bar{u}$ to (\ref{eq:ode in cylinder}) completely in Euclidean
space in \cite{CHY}. For $h<0$ and $k<\frac{n}{2}$, there exists a periodic function
$\xi(t)$ such that the metric $g=\frac{e^{2\xi(t)}}{|x|^{2}}g_{E}=\bar{u}_h(t)^{\frac{4}{n-2}}(dt^{2}+d\theta^{2})$
is complete in $\mathbb{R}^{n}\backslash\{0\}$ and $\bar{u}_h(t)$
is bounded with positive lower bound when $\sigma_{k}(g^{-1}A_{g})=c$.
Theorem \ref{Holder regularity} generalizes the classical result
for $k=1$ proved by Caffarelli-Gidas-Spruck \cite{CGS} and Korevaar-Mazzeo-Pacard-Schoen
\cite{KMPS}, and provides a powerful tool to deal with the singular
Yamabe problem.

Following the symbols in Pollack \cite{Pollack}, for a conformal
Killing field $X$ on $\mathbb{S}^{n}$ and $p\in P$, we denote 
\[
D_{k}(g,p)(X):=\int_{\Sigma}\mathring{H}_{a}^{b}\nu_{b}X^{a}d\sigma_{g}.
\]
Here $\Sigma$ is homologous to $\partial\overline{B_{\delta}(p)}$
and $\delta$ is sufficiently small such that $\overline{B_{\delta}(p)}$
contains no other singular points. For $k=1$, this number is called
the Dilational Pohozaev invariant. For general $k$, we also call
it $k$-Dilational Pohozaev invariant.

From the divergence structure of $<X,\nabla\sigma_{k}(g^{-1}A)>$,
we know that $D_{k}(g,p)(X)$ is independent of the hypersurface $\Sigma$,
which is homologous to $\partial \overline{B_{\delta}(p)}$. With Theorem \ref{thm:k-z}
and Theorem \ref{Holder regularity}, we can compute the local invariant
$D_{k}(g,p)(X)$ near the singularity $p$. By Theorem \ref{thm:k-z},
for $\mathbb{S}^{n}\backslash P$
\begin{equation}
\sum_{i=1}^{q}D_{k}(g,p_{i})(X)=0.\label{eq:equality}
\end{equation}

Let $X_{p}$ denote the conformal Killing field on $\mathbb{S}^{n}$ that fixes
the point $p$. Denote $D_{k}(g,p):=D_{k}(g,p)(X_{p}).$ 
\begin{lem}
\label{lem:k-invariant}Suppose $g=v^{\frac{4}{n-2}}g_{0}=u^{\frac{4}{n-2}}(dt^{2}+d\theta^{2})$
satisfying (\ref{eq:equation2}) around singular point $p$. Then there
exists a constant $h$ such that 
\[
D_{k}(g,p)=\binom{n-1}{k-1}\frac{n-k}{n}\frac{n}{2k}(\frac{1}{2})^{k-1}hw_{n-1},
\]
where $w_{n-1}$ is the volume of the unit $(n-1)$-sphere in $\mathbb{R}^{n}$
and $g$ is asymptotic to $g_{h}=\overline{u}_{h}^{\frac{4}{n-2}}(dt^{2}+d\theta^{2})$
in the sense that $|u(t,\theta)-\bar{u}_{h}(t)|\le Ce^{-\alpha t}$
for large $t$ as in Theorem \ref{Holder regularity}.
\end{lem}

Without confusion, we  always have the following standard relations. Consider the stereographic projection for $\mathbb{S}^{n}\backslash \{p\}$ from $\{p\}$ and $r_0<\pi$. Let $t_0=\ln \left(\tan(\frac{r_0}{2})\right)^{-1}$. 
The deleted geodesic ball $B_{r_0}(p)\backslash \{p\}$ is mapped to $$\mathbb{R}^{n}\backslash B_{\left(\tan(\frac{r_0}{2})\right)^{-1}}(0),$$ and corresponds to the half cylinder 
\[
C_{t_{0}}^{n}:=\{(t,\theta):\theta\in S^{n-1},t\ge t_{0}\}.
\] 
In the appendix, we will prove Lemma \ref{lem:k-invariant}.

From Han-Li-Teixeira \cite{HLT-1} and Chang-Han-Yang \cite{CHY},
we know the metric is conic for $k=n/2$. Here we say a metric $g$ conic if around
any singularity $p$, there exists a H\"older continuous function
$w(x)$ and a constant $\beta$ such that $g=e^{2u}g_{E}$ and $u(x)=\beta\log|x-p|+w(x)$.
For recent works about conic metrics for $\sigma_{2}$ Yamabe problem,
we refer to \cite{FW,FW2,FW3}. Especially for $k=2,n=4$, we assume
that at the origin, for $-1<\beta_{1}<0$, the radial solution to $(\ref{eq:euclidean equation})$ is $v(x)=|x|^{\beta_{1}}v_{1}(x)$
for a H\"older continuous radial function $v_{1}(x)$ and $g=v^{2}g_{E}$.
Since $\bar{u}_{h}(t)=|x|v(x)\rightarrow0$ as $|x|\rightarrow 0$,
from (\ref{eq:u_h definition}), we have $h=-(1-(\frac{\overline{u}_{h}'}{\overline{u}_{h}})^{2})^{2}=-\beta_{1}^{2}(2+\beta_{1})^{2}$.
This term also appears in Fang-Wei's subcritical condition in \cite{FW}.
From (\ref{eq:equality}) we know that the metric with constant $\sigma_{k}$
curvature can not admit one single singularity on $\mathbb{S}^{n}.$

\section{Growth estimate of singular solution}

In this section, we will describe the behavior of the solution near
the singularity by the contradiction argument. For reader's convenience,
we give the classical $C^{1}$ and $C^{2}$ estimates for $\sigma_{k}$
Yamabe equation as follows, which can be found in Theorem 1.20 in
\cite{LL1} and other papers such as \cite{GW1,SC,W1,L1}.
\begin{thm}
\label{thm:aprior estimates}Let $(M^{n},g)$ be a smooth, complete
Riemannian manifold for $n\geq3.$ Assume that $(M,g)$ has a positive
injectivity radius $i_{0}$ and $R_{ijkl}$, $|\nabla_{g}R_{ijkl}|,|\nabla_{g}^{2}R_{ijkl}|$
are bounded. On a geodesic ball $B_{3r}$ in $M$ of radius $3r\leq\frac{1}{2}i_{0},$
\[
\sigma_{k}\big(\lambda(A_{u^{\frac{4}{n-2}}g})\big)=h
\]
for any positive $h\in C^{2}(B_{3r})$ and $\lambda(A_{u^{\frac{4}{n-2}}g})\in\Gamma_{k}^{+}$.
Then for any positive $C^{4}$ solution $u$ on $B_{3r},$ we have
on $B_{r}$ 
\[
\big|\nabla_{g}(\log u)\big|_{g}(x)\leq C+C\left(\sup_{B_{2}}u^{\frac{2}{n-2}}\right)\sqrt{1+\sup_{B_{2r}}|\nabla h|},
\]
 and 
\[
\big|\nabla_{g}^{2}(\log u)\big|_{g}(x)\leq C+C\left(\sup_{B_{2r}}u^{\frac{4}{n-2}}\right)\left(1+\sup_{B_{2r}}|\nabla h|\right)+C\sup_{B_{2r}}\left(|\nabla h|^{2}+\left|\nabla^{2}h\right|\right),
\]
 where $C$ is some positive constant depending only on the upper
bound of $r,i_{0}$, $\sup_{B_{3r}}h$ and the bound of $R_{ijkl}$,
$|\nabla_{g}R_{ijkl}|,|\nabla_{g}^{2}R_{ijkl}|$.
\end{thm}

The following behavior near singularities in Theorem \ref{thm:Uniform decay}
is well known in the classical Yamabe problem and we refer to Schoen
and Pollack \cite{Pollack}. The local estimate for $\sigma_{k}$
Yamabe problem in $B_{r}\backslash\{0\}$ has also been  proved by
Li \cite{L}. 

\begin{thm}
\label{thm:Uniform decay}Suppose $g=u^{\frac{4}{n-2}}g_{0}$ with
smooth metric $g_{0}$ and $\sigma_{k}(g^{-1}A_{g})=C_{0}$, where
$C_{0}$ is a positive constant on $\mathbb{S}^{n}\backslash P$ with $\lambda(A_{g})\in\Gamma_{k}^{+}$.
Then 
\[
u(x)\le Cd(x,P)^{\frac{2-n}{2}},
\]
where $C$ is independent of $u$ and $d(\cdot,\cdot)$ denotes the
distance function on $(\mathbb{S}^{n},g_{0})$.
\end{thm}

In Theorem $1.1'$\cite{L}, Li proved this type of  estimate in local case
and the difference of our proof  from his  is that the constant $C$ here is
independent of $u$. Also Li-Nguyen have given a similar decay theorem
for a sequence of smooth solutions at the beginning of Part 3 for $\sigma_{k}$
Yamabe problem in \cite{LN}. See also \cite{LNW,FW2}.
\begin{proof}
Choose $x_{0}\in \mathbb{S}^{n}\backslash P$ and $\sigma$ sufficiently small
so that $\overline{B_{\sigma}(x_{0})}\subset \mathbb{S}^{n}\backslash P.$
Let $\rho(x)=d(x,x_{0})$ and define 
\[
f(x)=(\sigma-\rho(x))^{\frac{n-2}{2}}u(x).
\]

Since $u$ is smooth up to $\partial B_{\sigma}(x_{0})$, we have
$f=0$ on $\partial B_{\sigma}(x_{0})$. If $f(x)\le c$ for $x\in B_{\sigma}(x_{0})$,
choosing $\sigma=d(x_{0},P)/2$, the following holds 
\[
f(x_{0})=\sigma^{\frac{n-2}{2}}u(x_{0}),
\]
 and thus 
\[
u(x_{0})\le c2^{\frac{n-2}{2}}\cdot d(x_{0},P)^{\frac{2-n}{2}}.
\]
The theorem is proved. Otherwise, we have a sequence of functions $\{g_{i}=u_{i}^{\frac{4}{n-2}}g_{0}\}$
such that  for any $i$, there exists $x_i$ such that 

\[
(\sigma-d(x_{i},x_{0}))^{\frac{n-2}{2}}u_{i}(x_{i})=f(x_{i})=\max_{B_{\sigma}(x_{0})}f>i.
\]

Since $(\sigma-d(x_{i},x_{0}))^{\frac{n-2}{2}}\le\sigma^{\frac{n-2}{2}},$
we have $u_{i}(x_{i})\rightarrow\infty$ as $i\rightarrow\infty$.
Let $(y^{1},\cdots,y^{n})$ be a normal coordinate centered at $x_{i}$.
Denoting $\lambda_{i}:=u_{i}(x_{i})^{\frac{2}{n-2}},$ we get that
$\lambda_{i}\rightarrow\infty$ as $i\rightarrow\infty.$ Let $z=\lambda_{i}y$
and define
\[
\tilde{u}_{i}(z)=\lambda_{i}^{\frac{2-n}{2}}u_{i}(exp_{x_{i}}\frac{z}{\lambda_{i}}),
\]
where $\tilde{u}_{i}(0)=1$ and $exp$ is the exponential map. 

Let $g_{i}=\tilde{u}_{i}^{4/(n-2)}g_{0,i}$, where $g_{0,i}(z)=\sum_{k,l=1}^{n}g_{0,kl}(\frac{z}{\lambda_{i}})dz^{k}dz^{l}$
converges uniformly to the Euclidean metric and $g_{0}$ is centered
at $x_{i}$.

Denote $r_{i}:=\sigma-\rho(x_{i}).$ We have $u_{i}(x)\le2^{\frac{n-2}{2}}u_{i}(x_{i})$
in $B_{r_{i}/2}(x_{i}).$ So 
\[
\tilde{u}_{i}(z)\le2^{\frac{n-2}{2}}\quad\text{in}\quad|z|\le\lambda_{i}\frac{r_{i}}{2}=\frac{1}{2}(u_{i}(x_{i})(\sigma-\rho(x_{i}))^{\frac{n-2}{2}})^{\frac{2}{n-2}}\ge\frac{1}{2}i^{\frac{2}{n-2}}.
\]

Thus $\tilde{u}_{i}(z)$ is uniformly bounded on any compact set in
$\mathbb{R}^{n}$ and $\tilde{u}_{i}(0)=1$. Moreover, we know $\sigma_{k}(\lambda(A_{g_{i}}))=C_{0}$
in $B_{\lambda_{i}\frac{r_{i}}{2}}(0).$

By Theorem \ref{thm:aprior estimates} we know that for sufficiently
large $i$, there exists a constant $C>0$ independent of $i$ such
that on any compact set $K\subset\mathbb{R}^{n},$

\[
\sup_{K}|\nabla\ln\tilde{u}_{i}|+|\nabla^{2}\ln\tilde{u}_{i}|\le C.
\]

With $\tilde{u}_{i}(0)=1$, we have $\tilde{u}_{i}\ge C$ in $K$.
By the Schauder theory, we know 
\[
\sup_{K}|\ln\tilde{u}_{i}|_{C^{2,\alpha}}\le C.
\]

Then by Arzela-Ascoli theorem, there exists a subsequence of $\{\widetilde{u}_{i}\}$
(still denoted by $\widetilde{u}_{i}$) and a function $u_{\infty}$
such that $\tilde{u}_{i}\rightarrow u_{\infty}$ in $C^{2,\alpha}$
sense on any compact sets in $\mathbb{R}^{n}$. Therefore 
\[
\sigma_{k}(\lambda(A_{u_{\infty}^{\frac{4}{n-2}}g_{E}}))=C_{0}\quad{\rm in}\quad\mathbb{R}^{n}.
\]
Here $g_{\infty}=u_{\infty}^{\frac{4}{n-2}}g_{E}$ is the standard sphere metric,
which is proved by Li-Li\cite{LL2}. 

From the following theorem, we know that with respect to $g_{i}$, any ball
in $\mathbb{S}^{n}\backslash P$ has a concave boundary, which is contradicted
to the Liouville theorem proved by Li-Li \cite{LL2}.

Actually the $C$ in Theorem \ref{thm:Uniform decay} is also independent
of $d(x,P)$, which can be proved by replacing $x_{0}$ by $x_{0,i}$
in the above argument. 
\end{proof}
\begin{thm}
\label{thm:Sch}If $g$ is a complete metric of constant positive
$\sigma_{k}$ curvature on $\Omega\subset \mathbb{S}^{n}$( $\Omega\neq \mathbb{S}^{n}$),
which is conformal to $g_{0}$, then any ball $B$ (with respect to
$g_{0}$) with $\bar{B}\subset\Omega$ has geodesically convex boundary
$\partial B$ with respect to $g.$ (Here $g_{0}$ is the standard
sphere metric.)
\end{thm}

For the constant scalar curvature, Schoen \cite{Sc0} has proved this
theorem. For constant $Q$-curvature of order $\gamma$, inspired
by \cite{QR}, Gonz\'{a}lez-Mazzeo-Sire \cite{GMS} have proved
this type of theorem by the moving plane method. The moving plane method
was used in \cite{V4} to prove the Liouville theorem for $\sigma_{k}$
Yamabe problem. In the following, we will follow Gonz\'{a}lez-Mazzeo-Sire's
argument and give a brief description of the proof.
\begin{proof}
Denote the boundary of the geodesic ball $B_{r}(p)\subset\Omega$
by $S.$ By the stereographic projection for $x_{0}\in S$, with the
antipode on the plane, the boundary $S$ corresponds to the hyperplane
in $\mathbb{R}^{n}$ and we denote it as $H_{0}=\{(x_{1},\cdots,x_{n})|x_{n}=0\}$.
$\Omega$ becomes $\Omega_{1}$ and the singularities ($\partial\Omega_{1}$)
are below $H_{0}$, which are located on $\{(x_{1},\cdots,x_{n})|x_{n}<0\}$.
By the stereographic projection $g=v^{\frac{4}{n-2}}g_{E}$ and the
corresponding equation is, on $\Omega_{1}=\mathbb{R}^{n}\backslash\{p_{1},\cdots p_{q}\},$
\[
\sigma_{k}(-\frac{2}{n-2}v^{-1}\nabla^{2}v+\frac{2n}{(n-2)^{2}}v^{-2}dv\otimes dv-\frac{2}{(n-2)^{2}}v^{-2}|dv|^{2})=\binom{n}{k}(\frac{1}{2})^{k}v^{\frac{4k}{n-2}}.
\]

To prove that $S$ is geodesically convex with respect to $g$, we
only need to prove that the hyperplane $H_{0}$ is geodesically convex
with respect to $g=v^{\frac{4}{n-2}}g_{E}.$ The second fundamental
form of the hyperplane $H_{0}$ is 
\[
-\frac{4}{n-2}v^{-1}\frac{\partial v}{\partial x_{n}}I|_{H_{0}}.
\]

Next we use the moving plane method to show that $\frac{\partial v}{\partial x_{n}}<0$
on $H_{0},$ which yields the conclusion. Denote 
\[
v_{\lambda}(x)=v(x_{1},\cdots,x_{n-1},2\lambda-x_{n})
\]
 and $H_{\lambda}=\{(x_{1},\cdots,x_{n-1},x_{n})|x_{n}=\lambda\}$.
The proof is standard and we refer to \cite{V4} \cite{GMS} as well
as \cite{GNN} for more details. For completeness we just write part
of the proof.

We claim that there exists a constant $\lambda_{0}>0$ such that for
$\lambda>\lambda_{0}$, we have $w_{\lambda}(x)=v(x)-v_{\lambda}(x)<0$
on $\Sigma_{\lambda}=\{(x_{1},\cdots,x_{n})|x_{n}>\lambda\}$ and
$\frac{\partial w_{\lambda}}{\partial x_{n}}|_{H_{\lambda}}<0.$

Let $u(y)=v(\frac{y}{|y|^{2}})|y|^{2-n}$. When $y\rightarrow0,$
$\frac{y}{|y|^{2}}\rightarrow\infty$. For $|x|$ big enough $v(x)|x|^{n-2}$
is smooth and we know that $u(y)=v(\frac{y}{|y|^{2}})|y|^{2-n}$ is
smooth near $y=0$. By Taylor expansion, we have 
\[
u(y)=a_{0}+u_{i}(0)y_{i}+\frac{1}{2}u_{ij}(0)y_{i}y_{j}+o(|y|^{2}).
\]
For more details of Taylor expansion, we refer to \cite{V4}. 

Actually similar to the classical paper Gidas-Ni-Nirenberg \cite{GNN},
we have the following 

\[
v(x)=\frac{1}{|x|^{n-2}}(a_{0}+\frac{a_{i}x_{i}}{|x|^{2}}+\frac{a_{ij}x_{i}x_{j}}{|x|^{4}}+o(\frac{1}{|x|^{2}})).
\]

As shown in \cite{GNN,V4}, there exists a constant $\lambda_{0}>0$ such that
for $\lambda>\lambda_{0}$, we have $w_{\lambda}(x)=v(x)-v_{\lambda}(x)<0$
on $\Sigma_{\lambda}=\{(x_{1},\cdots,x_{n})|x_{n}>\lambda\}$. By
Hopf's lemma, on every compact sets of $H_{\lambda},$ $\frac{\partial w_{\lambda}}{\partial x_{n}}|_{H_{\lambda}}<0.$
By the argument of \cite{V4,GNN}, we know if there exists a $\lambda>0$
such that $w_{\lambda}<0$ in $\Sigma_{\lambda}$, then there exists
some $\varepsilon>0$ such that $w_{\lambda_{1}}<0$ for $\lambda_{1}\in[\lambda-\varepsilon,\lambda].$
Then the moving plane can be continued until touching the singularities.
Thus, $\frac{\partial w_{\lambda}}{\partial x_{n}}|_{H_{0}}=2\frac{\partial v}{\partial x_{n}}|_{H_{0}}<0.$
\end{proof}
By Theorem \ref{thm:aprior estimates} and Theorem \ref{thm:Uniform decay}
we get
\begin{cor}
\label{cor: high estimate} Suppose $\sigma_{k}(g^{-1}A_{g})$ is
a positive constant $\mathbb{S}^{n}\backslash P$, where $g=u^{\frac{4}{n-2}}g_{0}$
with $\lambda(A_{g})\in\Gamma_{k}^{+}$. For any compact subset $K\subset \mathbb{S}^{n}\backslash P$,
we have 
\[
|\nabla_{g_{0}}\ln u|_{C^{0}(K)}\le C,\quad|\nabla_{g_{0}}^{2}\ln u|_{C^{0}(K)}\le C,
\]
where $C$ depends on \textup{$K,dist(K,\mathbb{S}^{n}\backslash P),g_{0}.$} 
\end{cor}

\section{Proof of the Main theorem}

In this section, we give the proof of the main theorem.
\begin{proof}
By Corollary \ref{cor: high estimate}, $|u_{i}|_{C^{2}(K)}\le C_{K}$
on any compact set $K$ in $\mathbb{S}^{n}\backslash P:=\Omega.$ We know that
there exists a subsequence (still denoted as $u_{i}$) converging
to a function $u\in C^{1,1}(\mathbb{S}^{n}\backslash P)$ in $C_{loc}^{1,1}$
sense. We still need to prove that $u$ is actually smooth and positive
on $\mathbb{S}^{n}\backslash P.$

Assume that there exists a point $Q\in \mathbb{S}^{n}\backslash P$ such that
$u(Q)=0.$ Without loss of generality, let $\varepsilon_{i}=u_{i}(Q)$
and $\varepsilon_{i}\rightarrow0$ as $i\rightarrow\infty$. Define
$v_{i}:=\varepsilon_{i}^{-1}u_{i}$ and $\inf_{\mathbb{S}^{n}\backslash P}v_{i}=1.$
Since $u_{i}$ satisfies $\sigma_{k}(g_{u_{i}}^{-1}A_{g_{u_{i}}})=\binom{n}{k}(\frac{1}{2})^{k},$
we have 
\[
\sigma_{k}(g_{v_{i}}^{-1}A_{g_{v_{i}}})=\binom{n}{k}(\frac{1}{2})^{k}\varepsilon_{i}^{\frac{4k}{n-2}}.
\]

Since $u_{i}$ satisfies Lemma \ref{thm:Uniform decay} and Corollary
\ref{cor: high estimate}, for any $B_{2R}(y)\subset\subset \mathbb{S}^{n}\backslash P$
we have $\sup_{B_{R}(y)}u_{i}\le c\cdot\inf_{B_{R}(y)}u_{i},$ where
$c=c(n,R)$. Then $\sup_{B_{R}(y)}v_{i}\le c$ and $\ln v_{i}$ is
locally uniformly bounded on $\Omega.$ Now there exists a $v_{\infty}\in C^{1,1}(\Omega)$
such that $v_{i}\rightarrow v_{\infty}$ in $C_{loc}^{1,\alpha}(\Omega)$
by Theorem \ref{thm:aprior estimates}.\textcolor{red}{}

At $p_{1}\in P$, we can write the metric $g_{i}$ on cylindrical
coordinate. Here $D_{k}(g_{i},p_{1})$ has uniform positive lower
bound. Let $g_{i}=(wu_{i})^{4/(n-2)}g_{c}$ and $g_{0}=w^{4/(n-2)}g_{c}$,
where $g_{c}=dt^{2}+d\theta^{2}.$ For convenience, we denote $\{\frac{\partial}{\partial t},\frac{\partial}{\partial\theta_{i}}\}$
by $\{\frac{\partial}{\partial y_{1}},\cdots,\frac{\partial}{\partial y_{n}}\}$.
Denote 
\[
\bar{g}_{i}:=\varepsilon_{i}^{-4/(n-2)}g_{i}=(wv_{i})^{\frac{4}{n-2}}g_{c},
\]
and then
\[
\bar{g}_{i}\rightarrow g_{\infty}=(wv)^{\frac{4}{n-2}}g_{c}\quad\text{in}\quad C_{loc}^{1,\alpha}.
\]
As

\[
\bar{A}_{k}^{j}=\bar{g}^{jl}\bar{A}_{lk}=\bar{g}^{jl}(A_{lk})=\varepsilon_{i}^{4/(n-2)}A_{k}^{j},
\]
we have

\[
\mathring{\bar{H}}_{i,1}^{1}=\mathring{H}_{i,1}^{1}\varepsilon_{i}^{4k/(n-2)}.
\]

Therefore

\begin{align*}
\int_{\Sigma_{t_{0}}}\mathring{H}_{i,1}^{1}u_{i}^{2n/(n-2)}w^{2n/(n-2)}d\theta & =\int_{\Sigma_{t_{0}}}\mathring{\bar{H}}_{i,1}^{1}\varepsilon_{i}^{-4k/(n-2)}u_{i}^{2n/(n-2)}w^{2n/(n-2)}d\theta\\
 & =\int_{\Sigma_{t_{0}}}\mathring{\bar{H}}_{i,1}^{1}\varepsilon_{i}^{(2n-4k)/(n-2)}(v_{i}w)^{\frac{2n}{n-2}}d\theta,
\end{align*}
where 

\[
\int_{\Sigma_{t_{0}}}\mathring{\bar{H}}_{i,1}^{1}(v_{i}w)^{\frac{2n}{n-2}}d\theta
\]
 is bounded due to the local uniform $C^{2}$ estimate of $v_{i}$. 

For $k<n/2$, $\lim_{i\rightarrow\infty}\int_{\Sigma_{t_{0}}}\mathring{H}_{i,1}^{1}u_{i}^{2n/(n-2)}w^{2n/(n-2)}d\theta=0,$
which contradicts the uniform positive lower bound of $D_{k}(g_{i},p_{1})$.
Now we obtain that $u$ is positive on $\Omega$ and furthermore
$u_{i}$ has positive lower bound. Therefore by the classical Schauder
theory $u_{i}$ converges to $u$ in $C_{loc}^{\infty}$ sense and
$u$ satisfies $\sigma_{k}(\lambda(A_{u}))=\binom{n}{k}(\frac{1}{2})^{k}$
with non-vanishing $D_{k}(g_{u},p_{1})$, which is the limit of $D_{k}(g_{i},p_{1})$.
The non-vanishing $D_{k}(g_{u},p_{1})$ implies that $g_{u}$ has
non-removable singularities. With \cite{CHY,HLT-1}'s results, we
know that the metric is complete.
\end{proof}

\section{Appendix}

Following the argument for the scalar curvature by
Pollack \cite{Pollack}, we compute the accurate $D_{k}(g,p)$ by
Theorem \ref{Holder regularity}.
%\begin{lem}
%(Lemma \ref{lem:k-invariant})Suppose $g=v^{\frac{4}{n-2}}g_{0}$
%and $g$ satisfies (\ref{eq:equation2}) with singular point $p$.
%Then there exists a constant $h$ such that 
%\[
%D_{k}(g,p)=\binom{n-1}{k-1}\frac{n-k}{n}\frac{n}{2k}(\frac{1}{2})^{k-1}hw_{n-1},
%\]
%where $w_{n-1}$ is the volume of the unit $(n-1)$-sphere in $\mathbb{R}^{n}$,
%$h$ corresponds to $g_{h}=\overline{u}_{h}^{\frac{4}{n-2}}(dt^{2}+d\theta^{2})$
%and $g$ is asymptotic to $g_{h}.$
%\end{lem}
%
\begin{proof}[Proof of Lemma  \ref{lem:k-invariant} ]
Without loss of generality, we assume that $p$ is the north pole. Let
$(t,\theta)$ be the cylindrical coordinate about $p$ and $X_{p}=\frac{\partial}{\partial t}$.
In local coordinate, we write $g=u^{4/(n-2)}g_{c}:=u^{4/(n-2)}(dt^{2}+d\theta^{2})$. 

The deleted ball $B_{r_{0}}(p)\backslash\{p\}$ corresponds to 
\[
C_{t_{0}}^{n}:=\{(t,\theta):\theta\in \mathbb{S}^{n-1},t\ge t_{0}\}
\]
and we denote

\[
\Sigma_{t_{0}}:=\{(t,\theta):\theta\in \mathbb{S}^{n-1},t=t_{0}\}
\]
for some $t_{0}$. Define $\nu_{g}=u^{-2/(n-2)}\frac{\partial}{\partial t}$
and $d\sigma_{g}=u^{2(n-1)/(n-2)}d\theta$. We have 

\[
\mathring{H}_{a}^{b}\nu_{b}X^{a}=\mathring{H_{1}^{1}}\nu_{1}X^{1}=\mathring{H_{1}^{1}}u^{\frac{2}{n-2}},
\]
where

\begin{align*}
\mathring{H}_{1}^{1}= & H_{1}^{1}-\frac{H_{c}^{c}}{n}\\
= & \frac{1}{(k-1)!}\sum_{\substack{
i_{1},\cdots i_{k-1}=1\\
j_{1},\cdots,j_{k-1},l
}}^{n}\delta\left(\begin{array}{ccc}
i_{1} & \cdots & 1\\
j_{1} & \cdots & l
\end{array}\right)A_{i_{1}}^{j_{1}}\cdots A_{i_{k-1}}^{j_{k-1}}A_{1}^{l}\\
 & -\frac{1}{n}\frac{1}{(k-1)!}\sum_{\substack{
i_{1},\cdots i_{k-1},c=1\\
j_{1},\cdots,j_{k-1},l
}}^{n}\delta \left(\begin{array}{ccc}
i_{1} & \cdots & c\\
j_{1} & \cdots & l
\end{array}\right)A_{i_{1}}^{j_{1}}\cdots A_{i_{k-1}}^{j_{k-1}}A_{c}^{l}.
\end{align*}
From Theorem \ref{Holder regularity}, we know that there exists a
corresponding singular solution $\bar{u}_{h}(t)$ such that for some positive
$\alpha$
\[
u(t)=\bar{u}_{h}(t)+O(e^{-\alpha t}).
\]
Here $\bar{u}_{h}$ satisfies (\ref{eq:ode in cylinder}) and (\ref{eq:u_h definition}). For simplicity, we use $u_h$ to represent $\bar{u}_h$ in the following computations. 

By 
\[
A=A_{c}-\frac{2}{n-2}u^{-1}\nabla_{c}^{2}u+\frac{2n}{(n-2)^{2}}u^{-2}du\otimes du-\frac{2}{(n-2)^{2}}u^{-2}|du|_{g_{c}}^{2}g_{c},
\]

\[
A_{c}=-\frac{1}{2}dt^{2}+\frac{1}{2}d\theta^{2},
\]
we have 

\begin{align*}
 & A(\frac{\partial}{\partial t},\frac{\partial}{\partial t})\\
 & =-\frac{1}{2}-\frac{2}{n-2}u_{h}^{-1}\nabla_{c}^{2}u_{h}(\frac{\partial}{\partial t},\frac{\partial}{\partial t})+\frac{2n}{(n-2)^{2}}u_{h}^{-2}\frac{\partial u_{h}}{\partial t}\frac{\partial u_{h}}{\partial t}-\frac{2}{(n-2)^{2}}u_{h}^{-2}|du_{h}|_{g_{c}}^{2}+(u_{h}^{-2}+u_{h}^{-1})O(e^{-\alpha t})\\
 & =-\frac{1}{2}-\frac{2}{n-2}u_{h}^{-1}\frac{\text{\ensuremath{\partial^{2}u_{h}}}}{\partial t^{2}}+\frac{2n-2}{(n-2)^{2}}u_{h}^{-2}\frac{\partial u_{h}}{\partial t}\frac{\partial u_{h}}{\partial t}+(u_{h}^{-2}+u_{h}^{-1})O(e^{-\alpha t}),
\end{align*}

\[
A(\frac{\partial}{\partial t},\frac{\partial}{\partial\theta_{i}})=O(e^{-\alpha t})(u_{h}^{-1}+u_{h}^{-2}),
\]
and 

\[
A(\frac{\partial}{\partial\theta_{i}},\frac{\partial}{\partial\theta_{j}})=\frac{1}{2}\delta_{ij}-\frac{2}{(n-2)^{2}}u_{h}^{-2}(\frac{\partial u_{h}}{\partial t})^{2}\delta_{ij}+O(e^{-\alpha t})
\]
for $i,j=2,\cdots,n.$

Therefore 

\begin{align*}
 & \frac{1}{(k-1)!}\sum_{\substack{
i_{1},\cdots i_{k-1}=1\\
j_{1},\cdots,j_{k-1},l
}}^{n}\delta\left(\begin{array}{ccc}
i_{1} & \cdots & 1\\
j_{1} & \cdots & l
\end{array}\right)A_{i_{1}}^{j_{1}}\cdots A_{i_{k-1}}^{j_{k-1}}A_{1}^{l}\\
 & =O(e^{-\alpha t})+\frac{1}{(k-1)!}\sum_{\substack{
i_{1},\cdots i_{k-1}=1\\
j_{1},\cdots,j_{k-1}
}}^{n}\delta\left(\begin{array}{ccc}
i_{1} & \cdots & 1\\
i_{1} & \cdots & 1
\end{array}\right)A_{i_{1}}^{i_{1}}\cdots A_{i_{k-1}}^{i_{k-1}}A_{1}^{1}\\
 & =O(e^{-\alpha t})+\frac{1}{(k-1)!}\sum_{i_{1}\cdots i_{k-1}\neq1}^{n}\delta \left(\begin{array}{ccc}
i_{1} & \cdots & i_{k-1}\\
i_{1} & \cdots & i_{k-1}
\end{array}\right)A_{i_{1}}^{i_{1}}\cdots A_{i_{k-1}}^{i_{k-1}}A_{1}^{1}\\
 & =O(e^{-\alpha t})+A_{1}^{1}\sigma_{k-1}(g_{u_{h}}^{-1}A_{g_{u_{h}}}|_{i=2,\cdots n}),
\end{align*}
and 
%{\substack{k_0,k_1,\dots>0\\   k_0+k_1+\dots=n}}
\begin{align*}
 & -\frac{1}{n}\frac{1}{(k-1)!}\sum_{\substack{
i_{1},\cdots i_{k-1},c=1\\
j_{1},\cdots,j_{k-1},l
}}^{n}\delta \left(\begin{array}{ccc}
i_{1} & \cdots & c\\
j_{1} & \cdots & l
\end{array}\right)A_{i_{1}}^{j_{1}}\cdots A_{i_{k-1}}^{j_{k-1}}A_{c}^{l}\\
 & =O(e^{-\alpha t})-\frac{1}{n}\frac{1}{(k-1)!}\sum_{\substack{
i_{1},\cdots i_{k-1}=1\\
j_{1},\cdots,j_{k-1},c
}}^{n}\delta\left(\begin{array}{ccc}
i_{1} & \cdots & c\\
j_{1} & \cdots & c
\end{array}\right)A_{i_{1}}^{j_{1}}\cdots A_{i_{k-1}}^{j_{k-1}}A_{c}^{c}\\
 & =O(e^{-\alpha t})-\frac{k}{n}\bigg(\sigma_{k-1}(g_{u_{h}}^{-1}A_{g_{u_{h}}}|_{i=2,\cdots n})A_{1}^{1}+\sigma_{k}(g_{u_{h}}^{-1}A_{g_{u_{h}}}|_{i=2,\cdots,n})\bigg).
\end{align*}
Furthermore, we get 
\begin{align*}
\mathring{H}_{1}^{1}= & O(e^{-\alpha t})+A_{1}^{1}\sigma_{k-1}(g_{u_{h}}^{-1}A_{g_{u_{h}}}|_{i=2,\cdots n})(1-\frac{k}{n})-\frac{k}{n}\sigma_{k}(g_{u_{h}}^{-1}A_{g_{u_{h}}}|_{i=2,\cdots,n})\\
= & O(e^{-\alpha t})+(1-\frac{k}{n})u_{h}^{-\frac{4k}{n-2}}A_{11}\sigma_{k-1}(g_{c}^{-1}A_{g_{u}}|_{i=2,\cdots n})-\frac{k}{n}u_{h}^{-\frac{4k}{n-2}}\sigma_{k}(g_{c}^{-1}A_{g_{u_{h}}}|_{i=2,\cdots,n})\\
= & O(e^{-\alpha t})-\frac{k}{n}u_{h}^{-\frac{4k}{n-2}}C_{n-1}^{k}\left[\frac{1}{2}-\frac{2}{(n-2)^{2}}u_{h}^{-2}(\frac{\partial u_{h}}{\partial t})^{2}\right]^{k}\\
&+(1-\frac{k}{n})u_{h}^{-\frac{4k}{n-2}}\left(-\frac{1}{2}-\frac{2}{n-2}u_{h}^{-1}\frac{\text{\ensuremath{\partial^{2}u_{h}}}}{\partial t^{2}}+\frac{2n-2}{(n-2)^{2}}u_{h}^{-2}(\frac{\partial u_{h}}{\partial t})^{2}\right)\\
  &\times\left[\frac{1}{2}-\frac{2}{(n-2)^{2}}u_{h}^{-2}(\frac{\partial u_{h}}{\partial t})^{2}\right]^{k-1}C_{n-1}^{k-1}\\
= & O(e^{-\alpha t})+u^{-\frac{4k}{n-2}}C_{n-1}^{k-1}\frac{n-k}{n}\bigg[\frac{1}{2}-\frac{2}{(n-2)^{2}}u_{h}^{-2}(\frac{\partial u_{h}}{\partial t})^{2}\bigg]^{k-1}K_{h}(t),
\end{align*}
where $K_{h}(t)=-\frac{1}{2}-\frac{2}{n-2}u_{h}^{-1}\frac{\text{\ensuremath{\partial^{2}u_{h}}}}{\partial t^{2}}+\frac{2n-2}{(n-2)^{2}}u_{h}^{-2}(\frac{\partial u_{h}}{\partial t})^{2}-\frac{1}{2}+\frac{2}{(n-2)^{2}}u_{h}^{-2}(\frac{\partial u_{h}}{\partial t})^{2}.$

Without confusion, we sometimes use $u_{h,t}$, $u_{h,tt}$ as $\frac{\partial u_h}{\partial t}$ and $\frac{\partial^2 u_h}{\partial t^2}$ respectively.

As $u_{h}$ satisfies 
\[
\frac{1}{2}=\left(1-(\frac{2}{n-2})^{2}(\frac{u_{t}}{u})^{2}\right)^{k-1}\left(\frac{k}{n}\frac{2}{n-2}\left(-\frac{u_{tt}}{u}+\frac{u_{t}^{2}}{u^{2}}\right)+(\frac{1}{2}-\frac{k}{n})\left(1-(\frac{2}{n-2})^{2}\frac{u_{t}^{2}}{u^{2}}\right)\right)u^{-\frac{4k}{n-2}},
\]
we obtain 
\begin{equation}
-\frac{u_{h,tt}}{u_{h}}=\frac{n(n-2)}{2k}\left[\frac{1}{2}u_{h}^{\frac{4k}{n-2}}\left(1-(\frac{2}{n-2})^{2}(\frac{u_{h,t}}{u_{h}})^{2}\right)^{-k+1}-(\frac{u_{h,t}}{u_{h}})^{2}\left(\frac{2k}{n(n-2)}-\frac{n-2k}{2n}(\frac{2}{n-2})^{2}\right)+\frac{k}{n}-\frac{1}{2}\right].\label{eq:second derivative-1}
\end{equation}
By (\ref{eq:second derivative-1}), 
\begin{align*}
K_{h}(t)= & -1+\frac{2n}{(n-2)^{2}}u_{h}^{-2}(\frac{\partial u_{h}}{\partial t})^{2}\\
 & +\frac{n}{k}\bigg[\frac{1}{2}u_{h}^{\frac{4k}{n-2}}\left(1-(\frac{2}{n-2})^{2}(\frac{u_{h,t}}{u_{h}})^{2}\right)^{-k+1}-(\frac{u_{h,t}}{u_{h}})^{2}\left(\frac{2k}{n(n-2)}-\frac{n-2k}{2n}(\frac{2}{n-2})^{2}\right)+\frac{k}{n}-\frac{1}{2}\bigg]\\
= & -\frac{n}{2k}+\frac{n}{k}\frac{1}{2}u_{h}^{\frac{4k}{n-2}}\left(1-(\frac{2}{n-2})^{2}(\frac{u_{h,t}}{u_{h}})^{2}\right)^{-k+1}+\frac{u_{h,t}^{2}}{u_{h}^{2}}\frac{2n}{k(n-2)^{2}}.
\end{align*}
So 
\begin{align*}
\mathring{H}_{1}^{1}= & O(e^{-\alpha t})+C_{n-1}^{k-1}\frac{n-k}{n}\bigg\{-\frac{n}{2k}u_{h}^{-k(\frac{4}{n-2})}\times[\frac{1}{2}-\frac{2}{(n-2)^{2}}u_{h}^{-2}(\frac{\partial u_{h}}{\partial t})^{2}]^{k-1}+\frac{n}{2k}(\frac{1}{2})^{k-1}\\
 & +\frac{2n}{k(n-2)^{2}}\frac{u_{h,t}^{2}}{u_{h}^{2}}\left(\frac{1}{2}-\frac{2}{(n-2)^{2}}\frac{u_{h,t}^{2}}{u_{h}^{2}}\right)^{k-1}u_{h}^{-k(\frac{4}{n-2})}\bigg\}\\
= & O(e^{-\alpha t})+C_{n-1}^{k-1}\frac{n-k}{n}\bigg\{\left(\frac{1}{2}-\frac{2}{(n-2)^{2}}\frac{u_{h,t}^{2}}{u_{h}^{2}}\right)^{k-1}u_{h}^{-k(\frac{4}{n-2})}(-\frac{n}{2k}+\frac{2n}{k(n-2)^{2}}\frac{u_{h,t}^{2}}{u_{h}^{2}})+\frac{n}{2k}(\frac{1}{2})^{k-1}\bigg\}\\
= & O(e^{-\alpha t})+C_{n-1}^{k-1}\frac{n-k}{n}\frac{n}{2k}(\frac{1}{2})^{k-1}\bigg\{1-u_{h}^{-\frac{4k}{n-2}}\left(1-(\frac{2}{n-2})^{2}(\frac{u_{h,t}}{u_{h}})^{2}\right)^{k}\bigg\}.
\end{align*}
Furthermore, it holds that

\begin{align*}
 & \int_{\Sigma_{t_{0}}}\mathring{H}_{a}^{b}\nu_{b}X^{a}d\sigma_{g}\\
 & =\int_{\Sigma_{t_{0}}}\bigg[O(e^{-\alpha t})+C_{n-1}^{k-1}\frac{n-k}{n}\frac{n}{2k}(\frac{1}{2})^{k-1}\left(1-u_{h}^{-\frac{4k}{n-2}}\left(1-(\frac{2}{n-2})^{2}(\frac{u_{h,t}}{u_{h}})^{2}\right)^{k}\right)\bigg]u_{h}^{\frac{2n}{n-2}}\\
 & =\int_{\Sigma_{t_{0}}}\big[O(e^{-\alpha t})+C_{n-1}^{k-1}\frac{n-k}{n}\frac{n}{2k}(\frac{1}{2})^{k-1}h\big].
\end{align*}
As $D_{k}(p,g)$ is independent of $t_{0},$ we know 
\[
D_{k}(p,g)=C_{n-1}^{k-1}\frac{n-k}{n}\frac{n}{2k}(\frac{1}{2})^{k-1}hw_{n-1}.
\]
\end{proof}

\end{document}